\documentclass[12pt]{amsart}
\usepackage{amssymb,amsfonts,amsmath,amsthm,cite,verbatim}
\usepackage{xcolor}
\usepackage{tikz}
\usepackage{graphicx}
\usepackage[left=2.5cm, right=2.5cm, top=3.5cm, bottom=3.5cm]{geometry}
\usepackage[latin1]{inputenc} 
\usepackage{hyperref}

\usepackage[normalem]{ulem}

\numberwithin{equation}{section}
\newtheorem{theorem}{Theorem}[section]
\newtheorem{coro}[theorem]{Corollary}

\newtheorem{lemma}[theorem]{Lemma}

\newcommand{\abs}[1]{\lvert#1\rvert}

\begin{document}


\title{Rogers--Ramanujan type identities and Chebyshev
Polynomials of the third kind}

\author{Lisa H. Sun}

\address{Center for Combinatorics, LPMC,
Nankai University, Tianjin 300071, P.R. China}

\email{sunhui@nankai.edu.cn}


\begin{abstract}
It is known that $q$-orthogonal polynomials play an important role in the field of $q$-series and special functions. During studying Dyson's ``favorite" identity of Rogers--Ramanujan type, Andrews pointed out that the classical orthogonal polynomials also have surprising applications in the world of $q$. By inserting Chebyshev polynomials of the third  and the fourth kinds into Bailey pairs, Andrews derived a family of Rogers--Ramanujan type identities and also results related to mock theta functions and Hecke--type series. In this paper, by constructing a new Bailey pair involving Chebyshev polynomials of the third  kind, we further extend Andrews' way  in the studying of Rogers--Ramanujan type identities. By fitting this Bailey pair into different weak forms of Bailey's lemma, we obtain a companion identity to Dyson's favorite one and also many other Rogers--Ramanujan type identities. Furthermore, as immediate consequences, we also obtain some results related to Appell--Lerch series and the  generalized Hecke--type series.


\end{abstract}

\maketitle

\allowdisplaybreaks

\section{Introduction}\label{sec1}

 Freeman Dyson, in his article, A Walk Through Ramanujan's Garden \cite{Dys88},
describes his study of Rogers--Ramanujan type identities during the dark days of World War II. Among these identities  he found  his favorite one as follows
\begin{equation}\label{Dysonf}
\sum_{n\geq 0} \frac{q^{n^2+n}\prod_{j=1}^n(1+q^{j}+q^{2j})}{(q;q)_{2n+1}}=
\prod_{n=1}^\infty \frac{(1-q^{9n})}{(1-q^n)}.
\end{equation}
Dyson's proof of \eqref{Dysonf} and the proof subsequently provided
by Slater \cite[p. 161]{Sla52} are based on what has become known as Bailey's
Lemma \cite{Bai49}.

In the treatment of $q$-series, the $q$-orthogonal polynomials have been successfully applied to study different problems, especially to Rogers--Ramanujan type identities,  see, for example  \cite{And10, And12, AA77, ChFe13, IZ15}. Recently,  Andrews \cite{And19} pointed out that the classical orthogonal polynomials also could enter naturally into the world of $q$.

Denote the $n$th classical Chebyshev polynomials of the third kind by $V_n(x)$. By verifying the following identity \cite[Theorem 3.1]{And19} involving $V_n(x)$,
\begin{equation}\label{Andrewidentity}
\prod_{j=1}^n (1+2xq^j+q^{2j})=\sum_{j=0}^n q^{j+1\choose 2} V_j(x) {2n+1\brack n-j},
\end{equation}
Andrews obtained a Bailey pair
\[
\left(\frac{q^{n+1\choose 2}V_n(x)}{1-q}, \frac{\prod_{j=1}^n (1+2xq^j+q^{2j})}{(q;q)_{2n+1}}\right)
\]
at $a=q$. Then by fitting the above Bailey pair into a weak form of Bailey's lemma at $a=q$, Andrews \cite[(4.2)]{And19} derived the following generalization of Dyson's favorite identity \eqref{Dysonf}
\begin{equation}\label{AndV}
\sum_{n\geq 0} \frac{q^{n^2+n}\prod_{j=1}^n(1+2xq^{j}+q^{2j})}{(q;q)_{2n+1}}=
\frac{1}{(q;q)_\infty} \sum_{n\geq 0}q^{3{n+1\choose 2}}V_n(x),
\end{equation}
which reduces to many Rogers--Ramanujan type identities.


In this paper, we further apply Chebyshev polynomials of the third kind  to study a companion identity of Dyson's favorite one \eqref{Dysonf}
\begin{equation}\label{Entry534}
\sum_{n\geq 0} \frac{q^{2n^2}\prod_{i=1}^n (1+q^{2i-1}+q^{4i-2})}{(q^2;q^2)_{2n}}
=\frac{(q,q^5,q^6;q^6)_\infty(q^9;q^{18})_\infty}{(q;q)_\infty},
\end{equation}
which can be found in Ramanujan's lost notebook \cite[p.103, Entry 5.3.4]{AB09}.

By using Chebyshev polynomials of the third kind, we show that
\begin{equation}\label{kbaileypair}
\left(q^{n^2}\Big(V_n(x)+V_{n-1}(x)\Big),\frac{\prod_{j=1}^n (1+2xq^{2j-1}+q^{4j-2})}{(q^2;q^2)_{2n}}  \right)
\end{equation}
form a Bailey pair at $a=1$. Based on Bailey's Lemma \cite[p. 3, eq. (3.1)]{Bai49}, we obtain a generalization of \eqref{Entry534}  which reduces to several Rogers--Ramanujan type identities.

\begin{theorem}\label{mainthm} We have
\begin{equation}\label{maini}
\sum_{n\geq 0} \frac{q^{2n^2}\prod_{j=1}^n(1+2xq^{2j-1}+q^{4j-2})}{(q^2;q^2)_{2n}}=
\frac{1}{(q^2;q^2)_\infty} \sum_{n\geq 0}q^{3n^2}(V_n(x)+V_{n-1}(x)).
\end{equation}
\end{theorem}

Moreover, by fitting our key Bailey pair \eqref{kbaileypair} into different weak forms of Bailey's lemma, we are led to other identities in the similar form as \eqref{maini}. As their consequences, we obtain more Rogers--Ramanujan type identities and also some results involving Appell--Lerch series and the generalized Hecke--type series.

This paper is organized as follows. In Section 2, we recall some basic definitions and properties of Chebyshev polynomials of the third kind, Bailey pairs  and Bailey's lemma. In Section 3, we devote to construct our key Bailey pair involving  Chebyshev polynomials of the third kind. In Section 4, by fitting the Bailey pair into a weak form of Bailey's lemma, we derive \eqref{maini}. We also give the detailed procedures to obtain the companion identity \eqref{Entry534} when $x$ is taken to be $\frac{1}{2}$.  In Section 5, we consider another two weak forms of Bailey's lemma, from which we obtain more Rogers--Ramanujan type identities. In Section 6, by using Bailey's lemma, we restrict our  attention to results where Chebyshev polynomials have been inserted into the generalized Hecke--type series. In Section 7, we  study some of the Apell--Lerch series arising as immediate consequences of our main results. At last, in Section 8, we show the connection between our work and Andrews' result \eqref{Andrewidentity}, which leads to an identity on $q$-binomial coefficients by applying the orthogonality of Chebyshev polynomials of the third kind.

\section{Chebyshev polynomials and Bailey's lemma}

Throughout this paper, we adopt  standard notations and terminologies
for $q$-series \cite{GR04}. We assume that $\abs{q}<1$.
The $q$-shifted factorial is defined by
\[
(a;q)_n=\begin{cases}
1, & \text{\it if $n=0$}, \\[2mm]
(1-a)(1-aq)\cdots(1-aq^{n-1}), & \text{\it if $n\geq 1$}.
\end{cases}
\]
We also use the notation
\[
(a;q)_\infty=\prod_{n=0}^\infty (1-aq^n).
\]
There are more compact notations for the multiple $q$-shifted factorials:
\begin{align*}
(a_1,a_2,\dots,a_m;q)_n&=(a_1;q)_n(a_2;q)_n \cdots(a_m;q)_n,\\
(a_1,a_2,\dots,a_m;q)_{\infty}&=(a_1;q)_{\infty}(a_2;q)_{\infty}\cdots
(a_m;q)_{\infty}.
\end{align*}
The $q$-binomial coefficients, or Gaussian polynomials are given by
\[
{n \brack k} =
 \left\{
  \begin{array}{ll}
  0, & \hbox{if $k<0$ or $k>n$},\\
  \dfrac{(q;q)_n}
{(q;q)_k(q;q)_{n-k}}, & \hbox{otherwise.}
                        \end{array}
                      \right.
\]
We also denote the case when $q\to q^\ell$ by ${n\brack k}_{q^{\ell}}$.

It is known that $q$-orthogonal polynomials play an important role in the study of Rogers--Ramanujan type identities. In \cite{And19}, Andrews  pointed out that the classical orthogonal polynomials also can be naturally applied in the study of $q$-series.
Recall that the Chebyshev polynomial of the first kind is defined by
\[
T_n(x)=\cos n\theta,
\]
where $x=\cos \theta$. By combining the trigonometric identities, it is direct to derive that $T_n(x)$ satisfies the fundamental recurrence relation 
\[
T_n(x)=2xT_{n-1}(x)-T_{n-2}(x),
\]
for $n>1$ with the initial conditions
\[
T_0(x)=1, \ T_1(x)=x.
\]
Moreover, the Chebyshev polynomial of the third kind $V_n(x)$ is given by
\[
V_n(x)=\frac{\cos (n+\frac{1}{2})\theta}{\cos \frac{1}{2}\theta}
\]
which can be determined by
\begin{equation}\label{recV}
V_n(x)=2xV_{n-1}(x)-V_{n-2}(x),
\end{equation}
for $n>1$ together with the initial conditions $V_0(x)=1$ and $V_1(x)= 2x-1$. For  convenience, we also set   $V_{n}(x)=0$ for $n< 0$. These two kinds of Chebyshev polynomials are closely related. To be more precisely, we have for $n\geq 1$
\begin{equation} \label{CTV}
2T_n(x)=V_n(x)+V_{n-1}(x).
\end{equation}
Equivalently, we will present our results in terms of $V_n(x)$ in this paper.

In \cite[Lemma 4.1]{And19}, Andrews  stated some special values of $V_n(x)$ which can be easily derived by using the mathematical  induction based on the above recurrence relation \eqref{recV}.

\begin{lemma}\label{specialv} For $n\geq 0$,
\begin{subequations}
\begin{align}
V_n(-1)&=(-1)^n (2n+1), \label{V-1}\\[5pt]
V_n\Big(-\frac{1}{2}\Big)&=\left\{
                    \begin{array}{ll}
                      -2, & \text{if $n\equiv 1 \pmod{3}$,} \\
                      1, & \text{otherwise,}
                    \end{array}
                  \right.\label{V-12}\\[5pt]
V_n(0)&=\left\{
                    \begin{array}{ll}
                      1, & \text{if $n\equiv 0,3 \pmod{4}$,} \\
                      -1, & \text{otherwise,}
                    \end{array}
                  \right.\label{V-0}\\[5pt]
V_n\Big(\frac{1}{2}\Big)&=\left\{
                    \begin{array}{ll}
                      1, & \text{if $n\equiv 0,5 \pmod{6}$,} \\
 0, & \hbox{if $n\equiv 1,4 \pmod{6}$,} \\[5pt]
                      -1, & \text{if $n\equiv 2,3 \pmod{6}$,}
                    \end{array}
                  \right. \label{V12}\\[5pt]
V_n(1)&=1,\label{V1}\\[5pt]
V_n\Big(\frac{3}{2}\Big)&=F_{2n+1},\label{V23} \\[5pt]
V_n\Big(-\frac{3}{2}\Big)&=(-1)^nL_{2n+1},\label{V-23}
\end{align}
where $F_n$ and $L_n$ are the Fibonacci and Lucas numbers which are defined by the recurrence relations
\begin{align*}
 F_n&=F_{n-1}+F_{n-2},\\[5pt]
 L_n &= L_{n-1} + L_{n-2}
\end{align*}
for $n>1$ combined with the initial values $F_0=0, F_1=1$ and $L_0=2, L_1=1$, respectively.
\end{subequations}
\end{lemma}

A popular method to prove identities of Rogers--Ramanujan type is based
on Bailey's lemma, see \cite{And86,Bai49, JS17, War01}.
During studying  Rogers' work on Ramanujan's identities  \cite{AB05, AB09,  Rog94, Rog17}, Bailey \cite{Bai49} discovered the underlying mechanism which was named ``Bailey transform".   The most famous specialization of Bailey transformation now known as ``Bailey pair" which is given as a pair of sequence of rational functions $(\alpha_n(a,q), \beta_n(a,q))_{n\geq 0}$  with respect to $a$ such that
\begin{equation}\label{Baileyp}
\beta_n(a,q)=\sum_{j=0}^n \frac{\alpha_j(a,q)}{(q;q)_{n-j}(aq;q)_{n+j}}.
\end{equation}
In this paper, our key Bailey pair is obtained by fitting $V_n(x)$ into an identity of the above form.

Bailey \cite{Bai49} provided  the following fundamental result for producing an infinite family of identities out of one identity, see also Andrews \cite[pp. 25--27, Theorem 3.3]{And86}.

\begin{lemma}[Bailey's lemma] \label{baileylem} If $\alpha_n(a,q), \beta_n(a,q)$ form a Bailey pair, then
\begin{align}\label{baileylemi}
&\frac{1}{(aq/\rho_1,aq/\rho_2;q)_n}\sum_{j=0}^n \frac{(\rho_1,\rho_2;q)_j(aq/\rho_1\rho_2;q)_{n-j}}{(q;q)_{n-j}} \bigg(\frac{aq}{\rho_1\rho_2}\bigg)^j \beta_j(a,q) \nonumber \\[5pt]
&\qquad =\sum_{j=0}^n \frac{(\rho_1,\rho_2;q)_j}{(q;q)_{n-j}(aq;q)_{n+j}(aq/\rho_1,aq/\rho_2;q)_j} \bigg(\frac{aq}{\rho_1\rho_2}\bigg)^j \alpha_j(a,q).
\end{align}

\end{lemma}

There are some special weak forms of Bailey's lemma which attracts more attention since they are more direct to obtain Rogers--Ramanujan type identities from Bailey pairs.  By collecting a list of 96 Bailey pairs, and using some weak forms of Bailey's lemma, Slater compiled her famous list of 130 identities of  Rogers--Ramanujan type \cite{Sla51, Sla52}. We are mainly concerned the following four weak forms of Bailey's lemma.

\begin{lemma}\label{weakbailey} We have
\begin{subequations}
\begin{align}
&\sum_{n\geq 0} q^{n^2} \beta_n(1,q)=\frac{1}{(q;q)_\infty} \sum_{n\geq 0} q^{n^2}\alpha_n(1,q), \label{baileypair1}\\
&\sum_{n\geq 0} q^{n^2} (-q;q^2)_n\beta_n(1,q^2)=\frac{(-q;q^2)_\infty}{(q^2;q^2)_\infty} \sum_{n\geq 0} q^{n^2}\alpha_n(1,q^2), \label{baileypair2}\\
&2\sum_{n\geq 0} (-1)^n (q;q^2)_n\beta_n(1,q)=\frac{(q;q^2)_\infty}{(q^2;q^2)_\infty} \sum_{n\geq 0} (-1)^n \alpha_n(1,q),  \label{baileypair3}\\
&\sum_{n\geq 0} q^{n(n+1)/2}(-1;q)_n \beta_n(1,q)=\frac{(-q;q)_\infty}{(q;q)_\infty} \sum_{n\geq 0} \frac{q^{n(n+1)/2}(-1;q)_n}{(-q;q)_n} \alpha_n(1,q).\label{baileypair4}
\end{align}
\end{subequations}
\end{lemma}

These four weak forms can be obtained from Bailey's Lemma \ref{baileylem} by taking $a=1$, $n\to \infty$, and  $\rho_1, \rho_2$ to be certain special values. More precisely,   \eqref{baileypair1} is obtained  by setting $\rho_1,\rho_2\to \infty$, \eqref{baileypair2} is derived  by taking $q\to q^2, \rho_1\to \infty, \rho_2\to -q$, \eqref{baileypair3} is followed by setting $\rho_1\to \sqrt{q}, \rho_2\to -\sqrt{q}$, and \eqref{baileypair4} is derived by taking  $\rho_1\to \infty, \rho_2\to -1$. For more details, see, for example,  \cite{GL16,LSZ08,Sil18}.

Moreover, by applying Bailey's lemma
iteratively to an appropriate Bailey pair in the simple sum case, one can obtain multi-analog identities of Rogers--Ramanujan type straightforwardly. Let us take the following one \cite{And84} as an illustration.

\begin{lemma} Let $(\alpha_n(a,q), \beta_n(a,q))$ be a Bailey pair, then
\begin{align}
&\sum_{n_k\geq n_{k-1}\geq \cdots \geq n_1\geq 0}  \frac{a^{n_1+\cdots +n_k}q^{n_1^2+\cdots +n_k^2}\beta_{n_1}(a,q)}{(q)_{n_k-n_{k-1}}(q)_{n_{k-1}-n_{k-2}}\cdots (q)_{n_2-n_1}}\nonumber \\[5pt]
&\qquad =\frac{1}{(aq)_\infty} \sum_{n\geq 0} q^{kn^2} a^{kn} \alpha_n(a,q). \label{baileypairm}
\end{align}
\end{lemma}
Obviously, when $a=1$ and $k=1$, the above identity reduces to \eqref{baileypair1}.

\section{The key Bailey pair involving $V_n(x)$}

The object of this section is to construct the Bailey pair which bring Chebyshev polynomials of the third kind $V_n(x)$  into the field of $q$-series. By inserting this Bailey pair into the weak forms of Bailey's lemma, we will obtain many Rogers--Ramanujan type identities.

Based on the three term recurrence relation \eqref{recV} of $V_n(x)$, we obtain the following result.

\begin{theorem}\label{basici} We have
\begin{equation}\label{baileyid}
\prod_{j=1}^n (1+2xq^{2j-1}+q^{4j-2})=\sum_{j=0}^n q^{j^2}{2n \brack n-j}_{q^2}\Big(V_j(x)+V_{j-1}(x)\Big).
\end{equation}
\end{theorem}

\proof For brevity, we write $v_n(x)=V_{n}(x)+V_{n-1}(x)$. Obviously, $v_0(x)=1$, $v_1(x)=V_1(x)+V_0(x)=2x$, and $v_2(x)=V_2(x)+V_1(x)=4x^2-2$. It is easy to see that  $\{v_n(x)\}_{n\geq 0}$ form a basis for the polynomials in $x$ over $\mathbb{C}$.

Since for $n>1$, $V_n(x)$ satisfies the three-term recurrence relation
\[
V_n(x)=2xV_{n-1}(x)-V_{n-2}(x),
\]
It is easy to show that, for $n>2$,
\begin{equation}\label{Arec}
v_n(x)=2xv_{n-1}(x)-v_{n-2}(x).
\end{equation}

Denote the left and the right hand sides of \eqref{baileyid}  by $L_n(x)$ and $R_n(x)$, respectively. Notice that $L_n(x)$ is uniquely determined by the recurrence relation
\begin{equation}\label{recL}
L_{n}(x)=(1+2xq^{2n-1}+q^{4n-2})L_{n-1}(x)
\end{equation}
for $n\geq 1$ combined with $L_0(x)=1$.
Clearly, $R_0(x)=1$. Therefore, to show that $L_n(x)=R_n(x)$, it is sufficient to prove that  for $n>0$, $R_n(x)$ satisfies the same recurrence relation \eqref{recL}, which can be rewritten as follows
\begin{equation}\label{recR}
2xq^{2n-1}R_{n-1}(x)=R_{n}(x)-(1+q^{4n-2})R_{n-1}(x).
\end{equation}

By \eqref{Arec}, it directly leads to that for $j>1$
\[
2xv_j(x)=v_{j+1}(x)+v_{j-1}(x).
\]
Substituting the above relation into \eqref{recR}, it becomes the following form
\begin{align*}
&q^{2n-1}\left({2n-2 \brack n-1}_{q^2} 2x+q {2n-2 \brack n-2}_{q^2} 2x v_1(x)+ \sum_{j=2}^{n-1} q^{j^2}{2n-2 \brack n-1-j}_{q^2}(v_{j+1}(x)+v_{j-1}(x))\right)\\
&\qquad =\sum_{j\geq 0} q^{j^2} \left({2n \brack n-j}_{q^2}-(1+q^{4n-2}){2n-2 \brack n-1-j}_{q^2}\right)v_j(x).
\end{align*}
Then in the first two terms on the left hand side of the above identity, we can replace $2x$ and $2xv_1(x)$ with $v_1(x)$   and  $v_2(x)+2v_0(x)$, respectively. Since
$\{v_n(x)\}_{n\geq 0}$ form a basis for the polynomials in $x$, to verify $L_n(x)=R_n(x)$, it is sufficient to prove the coefficients of $v_j(x)$ on both sides of the above identity coincide with each other, that is, for $j\geq 0$,
\begin{align*}
&q^{2n-1+(j-1)^2}{2n-2 \brack n-j}_{q^2}+q^{2n-1+(j+1)^2}{2n-2 \brack n-2-j}_{q^2}\\
&\qquad =q^{j^2} {2n \brack n-j}_{q^2}-q^{j^2}(1+q^{4n-2}){2n-2 \brack n-1-j}_{q^2},
\end{align*}
which can be confirmed by direct simplification, and thereby \eqref{baileyid} is valid. \qed

Now, if we rewrite  \eqref{baileyid} as follows
\[
\frac{\prod_{j=1}^n (1+2xq^{2j-1}+q^{4j-2})}{(q^2;q^2)_{2n}}=\sum_{j=0}^n \frac{q^{j^2}(V_j(x)+V_{j-1}(x))}{(q^2;q^2)_{n+j}(q^2;q^2)_{n-j}},
\]
then by \eqref{Baileyp}, it immediately implies our key Bailey pair
\begin{equation}\label{baileypairV}
\left(q^{n^2}\Big(V_n(x)+V_{n-1}(x)\Big),\frac{\prod_{j=1}^n (1+2xq^{2j-1}+q^{4j-2})}{(q^2;q^2)_{2n}}  \right)
\end{equation}
relative to $a=1$ and $q\to q^2$.
By fitting this Bailey pair \eqref{baileypairV} into the weak forms of Bailey's Lemma \ref{weakbailey}, we will obtain a family of Rogers--Ramanujan type identities and also identities related to Apell--Lerch series and Hecke--type series.

\section{The weak form \eqref{baileypair1} of Bailey's Lemma}

In this section, we will show that how to derive the companion identity \eqref{Entry534} of Dyson's favourite one \eqref{Dysonf} by using the key Bailey pair \eqref{baileypairV} associated with $V_n(x)$. Meanwhile, by taking $x$ to be some other special values, we will  obtain more Rogers--Ramanujan type identities. We also consider the multisum generalization of these identities.

First, by setting $q\to q^2$ in the weak form  \eqref{baileypair1} of Bailey's Lemma and with aid of  the Bailey pair \eqref{baileypairV}, we directly obtain  Theorem \ref{mainthm}. Now, we will show  how to derive Ramanujan's identity \eqref{Entry534} from Theorem \ref{mainthm}.

\begin{theorem}
Ramanujan's identity \eqref{Entry534} is valid.
\end{theorem}

\proof Denote the left hand side of identity \eqref{Entry534} by $L$. In Theorem \ref{mainthm}, by taking $x=\frac{1}{2}$ and using the special value of $V_n(x)$   \eqref{V12}, we have
\begin{align*}
L&=\sum_{n\geq 0}  \frac{q^{2n^2}\prod_{j=1}^n(1+q^{2j-1}+q^{4j-2})}{(q^2;q^2)_{2n}}\\[5pt]
&=\frac{1}{(q^2;q^2)_\infty}\left(1+ \sum_{n\geq 1} q^{3n^2}\bigg(V_n\Big(\frac{1}{2}\Big)+V_{n-1}\Big(\frac{1}{2}\Big)\bigg)\right)\\[5pt]
&=\frac{1}{(q^2;q^2)_\infty} \Big(1+2\sum_{n\geq 1} q^{3(6n)^2}+\sum_{n\geq 0} q^{3(6n+1)^2}-\sum_{n\geq 0} q^{3(6n+2)^2}\\[5pt]
&\qquad\qquad\qquad -2\sum_{n\geq 0} q^{3(6n+3)^2}-\sum_{n\geq 0} q^{3(6n+4)^2}+\sum_{n\geq 0} q^{3(6n+5)^2}\Big).
\end{align*}
Taking the parity of $n$ into consideration, we see that \begin{align*}
&\sum_{n\geq 1} q^{3(6n)^2}-\sum_{n\geq 0} q^{3(6n+3)^2}=\sum_{n\geq 1} (-1)^n q^{3(3n)^2},\\[5pt]
&\sum_{n\geq 0} q^{3(6n+1)^2}-\sum_{n\geq 0} q^{3(6n+4)^2}=\sum_{n\geq 0} (-1)^n q^{3(3n+1)^2},\\[5pt]
&\sum_{n\geq 0} q^{3(6n+2)^2}-\sum_{n\geq 0} q^{3(6n+5)^2}=\sum_{n\geq 0} (-1)^n q^{3(3n+2)^2}.
\end{align*}
Therefore, it implies that
\begin{align}
L&=\frac{1}{(q^2;q^2)_\infty}\Big(1+2\sum_{n\geq 1} (-1)^nq^{3(3n)^2}+\sum_{n\geq 0} (-1)^n  q^{3(3n+1)^2}-\sum_{n\geq 0} (-1)^n  q^{3(3n+2)^2}\Big)\nonumber\\[5pt]
&  = \frac{1}{(q^2;q^2)_\infty}\Big(\sum_{n=-\infty}^\infty (-1)^nq^{3(3n)^2}+\sum_{n=-\infty}^\infty (-1)^n  q^{3(3n+1)^2}\Big)\nonumber\\[5pt]
&=\frac{1}{(q^2;q^2)_\infty}\sum_{n=-\infty}^\infty (-1)^n q^{3n^2}e^{\frac{2n\pi i}{3}}, \label{p534-1}
\end{align}
in which the last step can be affirmed  by considering the summation by  congruences of $n$ module 3. Then by applying the famous Jacobi's triple product identity which is  for $z, q \in \mathbf{C}$, $z\neq 0$ and $|q|<1$,
\begin{equation}\label{jacobi}
\sum_{n=-\infty}^\infty (-1)^n q^{n\choose 2} z^n=(q,z,q/z;q)_\infty,
\end{equation}
we further obtain that
\begin{align*}
L&=\frac{1}{(q^2;q^2)_\infty} (q^3e^{\frac{2\pi i}{3}},q^3e^{-\frac{2\pi i}{3}},q^6;q^6)_\infty \\[5pt]
&=\frac{(q^6;q^6)_\infty}{(q^2;q^2)_\infty} \prod_{n\geq 0} (1-q^{3+6n}e^{\frac{2\pi i}{3}})(1-q^{3+6n}e^{-\frac{2\pi i}{3}})\\[5pt]
&=\frac{(q^6;q^6)_\infty}{(q^2;q^2)_\infty} \prod_{n\geq 0} \big(1-q^{3+6n}(e^{\frac{2\pi i}{3}}+e^{-\frac{2\pi i}{3}})+q^{6+12n}\big)\\[5pt]
&=\frac{(q^6;q^6)_\infty}{(q^2;q^2)_\infty} \prod_{n\geq 0} (1+q^{3+6n}+q^{6+12n})\\[5pt]
&=\frac{(q^6;q^6)_\infty}{(q^2;q^2)_\infty} \prod_{n\geq 0}\frac{(1-q^{9+18n})}{(1-q^{3+6n})} \\[5pt]
&=\frac{(q^6;q^6)_\infty(q^9;q^{18})_\infty}{(q^2;q^2)_\infty(q^3;q^6)_\infty},
\end{align*}
which completes the proof by multiplying  both of the numerator and the denominator by $(q,q^5;q^6)_\infty$ and then simplifying. \qed

Noting that by setting $x=-\frac{1}{2}$ in \eqref{maini} and following the similar procedures as above, we  can also  get \eqref{Entry534} after  substituting $q\rightarrow -q$.

As more consequences of Theorem \ref{mainthm}, let us consider cases corresponding to other special values of $x$ as given in  Lemma \ref{specialv}.
The special value of
$V_n(x)$ at $x=-1$ (or equivalently, $x=1$)  yields the following Ramanujan's identity.

\begin{theorem}[Entry 5.3.3, {\cite[P. 102]{AB09}}] We have
\begin{align}
&\sum_{n\geq 0} \frac{q^{2n^2}(q;q^2)_n^2 }{(q^2;q^2)_{2n}}=
\frac{(q^3;q^3)_\infty (q^3;q^6)_\infty}{(q^2;q^2)_\infty}. \label{x-1}
\end{align}
\end{theorem}

\proof  Taking $x=-1$ in \eqref{maini} and using the special value of $V_n(x)$  at $x=-1$ \eqref{V-1}, we have
\begin{align*}
\sum_{n\geq 0} \frac{q^{2n^2}\prod_{j=1}^n (1-2q^{2j-1}+q^{4j-2})}{(q^2;q^2)_{2n}}&=\frac{1}{(q^2;q^2)_\infty} \sum_{n\geq 0} q^{3n^2}\big(V_n(-1)+V_{n-1}(-1)\big)\\[5pt]
&=\frac{1}{(q^2;q^2)_\infty}\big(1+2\sum_{n\geq 1} (-1)^nq^{3n^2}\big)\\[5pt]
&=\frac{1}{(q^2;q^2)_\infty}\sum_{n=-\infty}^\infty (-1)^n q^{3n^2}.
\end{align*}
Then the proof is complete by using Jacobi's triple product identity \eqref{jacobi} and simplifying. \qed

When $x$ is taken to be zero, we obtain Entry 5.3.2 in Ramanujan's lost notebook.

\begin{theorem}[Entry 5.3.2, {\cite[P. 101]{AB09}}] We have
\begin{align}\label{Entry532}
\sum_{n\geq 0} \frac{q^{n^2}(-q;q^2)_n }{(q;q)_{2n}}=\frac{(q^6;q^{12})_\infty(q^6;q^6)_\infty}{(q;q)_\infty}.
\end{align}
\end{theorem}

\proof By setting $x=0$ in \eqref{maini} and using the special value of $V_n(x)$  at $x=0$ \eqref{V-0},  we obtain
\begin{align*}
\sum_{n\geq 0} \frac{q^{2n^2}\prod_{j=1}^n (1+q^{4j-2})}{(q^2;q^2)_{2n}}&=\frac{1}{(q^2;q^2)_\infty} \sum_{n\geq 0} q^{3n^2}\big(V_n(0)+V_{n-1}(0)\big)\\[5pt]
&=\frac{1}{(q^2;q^2)_\infty}\Big(1+2\sum_{n\geq 1} q^{3(4n)^2}-2\sum_{n\geq 0} q^{3(4n+2)^2}\Big)\\[5pt]
&=\frac{1}{(q^2;q^2)_\infty}\Big(1+2\sum_{n\geq 1} (-1)^nq^{3(2n)^2}\Big)\\[5pt]
&=\frac{1}{(q^2;q^2)_\infty}\sum_{n=-\infty}^\infty (-1)^n q^{12n^2}.
\end{align*}
By  Jacobi's triple product identity \eqref{jacobi}, it leads to that
\begin{align*}
\sum_{n\geq 0} \frac{q^{2n^2}(-q^2;q^4)_n}{(q^2;q^2)_{2n}}&=\frac{(q^{12},q^{12},q^{24};q^{24})_\infty}{(q^2;q^2)_\infty}\\
&=\frac{(q^{12};q^{24})_\infty(q^{12};q^{12})_\infty}{(q^2;q^2)_\infty},
\end{align*}
which completes the proof by replacing  $q$ with $q^{\frac{1}{2}}$ in the above identity. \qed

For identity \eqref{Entry532}, it is also contained in Slater's list \cite[p.155, (29)]{Sla52}, and one can see also Andrews and Berndt \cite[p. 254, Entry 11.3.1]{AB05}.
For the above two identities \eqref{x-1} and  \eqref{Entry532}, Andrews  also considered the theta expansions of their left hand sides, which were given by equations $(3.1)_R$ and $(3.2)_R$ in \cite{And81}, respectively.

Taking $x=\frac{3}{2}$ and $-\frac{3}{2}$ in \eqref{maini}, respectively, we obtain the following two identities immediately.

\begin{theorem} We have
\begin{align}
& \sum_{n\geq 0} \frac{\prod_{j=1}^n (1+3q^{2j-1}+q^{4j-2})q^{2n^2}}{(q^2;q^2)_{2n}}=
\frac{1}{(q^2;q^2)_\infty}\sum_{n\geq 0} q^{3n^2}(F_{2n+1}+F_{2n-1}),\label{x32}\\[5pt]
& \sum_{n\geq 0} \frac{\prod_{j=1}^n (1-3q^{2j-1}+q^{4j-2})q^{2n^2}}{(q^2;q^2)_{2n}}=
\frac{1}{(q^2;q^2)_\infty}\Big(1+\sum_{n\geq 1} (-1)^n q^{3n^2}L_{2n}\Big). \label{x-32}
\end{align}
\end{theorem}

Remark that from \cite[A000032]{Slo64}, we see that for $n\geq 1$,  $L_n=2F_{n+1}-F_{n}$. It clearly implies that for $n\geq 1$
\begin{equation}\label{eqLF}
L_{2n}=F_{2n+1}+F_{2n-1}.
\end{equation}
Consequently, the identities \eqref{x32} and \eqref{x-32} are equivalent by substituting $q\to -q$.

Now, let us consider the multi-analog of Rogers--Ramanujan type identities. By inserting the key Bailey pair \eqref{baileypairV} into  \eqref{baileypairm} with $a=1$ and $q\to q^2$, we obtain the following generalization of Theorem \ref{mainthm}.

\begin{theorem} We have
\begin{align}
&\frac{1}{(q^2;q^2)_\infty} \sum_{n\geq 0} q^{(2k+1)n^2} (V_n(x)+V_{n-1}(x)) \nonumber\\
&\quad =\sum_{n_k\geq n_{k-1}\geq \cdots \geq n_1\geq 0} \frac{q^{2(n_1^2+\cdots +n_k^2)}\prod_{j=1}^{n_1} (1+2xq^{2j-1}+q^{4j-2})}{(q^2;q^2)_{n_k-n_{k-1}}(q^2;q^2)_{n_{k-1}-n_{k-2}}\cdots (q^2;q^2)_{n_2-n_1}(q^2;q^2)_{2n_1}}.  \label{mainthm3}
\end{align}
\end{theorem}

Obviously, when $k=1$, we are led to  Theorem \ref{mainthm} immediately. Note that the summation on the left hand side of the above identity can be obtained by substituting $q\to q^{\frac{2k+1}{3}}$ into the right hand side of identity \eqref{maini}. Therefore, based on Theorems 4.1--4.4 and by taking $x$ to be  $0$ (with $q^2\to q$), $1/2$ (or equivalently $-1/2$), $1$ (or equivalently $-1$), and $3/2$ (or equivalently $-3/2$), respectively, we can obtain the following identities immediately.

\begin{coro} We have
\begin{subequations}
\begin{align}
&\frac{(q^{4k+2},q^{4k+2},q^{8k+4};q^{8k+4})_\infty}{(q;q)_\infty} \nonumber\\
&\quad =\sum_{n_k\geq n_{k-1}\geq \cdots \geq n_1\geq 0} \frac{q^{n_1^2+\cdots +n_k^2}(-q;q^2)_{n_1}}{(q;q)_{n_k-n_{k-1}}(q;q)_{n_{k-1}-n_{k-2}}\cdots (q;q)_{n_2-n_1}(q;q)_{2n_1}}, \label{mainthm3-3}\\[8pt]
&\frac{(q^{4k+2};q^{4k+2})_\infty (q^{6k+3};q^{12k+6})_\infty}{(q^2;q^2)_\infty(q^{2k+1};q^{4k+2})_\infty} \nonumber\\[5pt]
&\quad =\sum_{n_k\geq n_{k-1}\geq \cdots \geq n_1\geq 0} \frac{q^{2(n_1^2+\cdots +n_k^2)}(q^3;q^6)_{n_1}}{(q^2;q^2)_{n_k-n_{k-1}}(q^2;q^2)_{n_{k-1}-n_{k-2}}\cdots (q^2;q^2)_{n_2-n_1}(q^2;q^2)_{2n_1}(q;q^2)_{n_1}},  \label{mainthm3-2}\\[8pt]
&\frac{(q^{2k+1},q^{2k+1},q^{4k+2};q^{4k+2})_\infty}{(q^2;q^2)_\infty} \nonumber\\[5pt]
&\quad =\sum_{n_k\geq n_{k-1}\geq \cdots \geq n_1\geq 0} \frac{q^{2(n_1^2+\cdots +n_k^2)}(q;q^2)_{n_1}^2}{(q^2;q^2)_{n_k-n_{k-1}}(q^2;q^2)_{n_{k-1}-n_{k-2}}\cdots (q^2;q^2)_{n_2-n_1}(q^2;q^2)_{2n_1}},  \label{mainthm3-1}\\[8pt]
&\frac{1}{(q^2;q^2)_\infty} \sum_{n\geq 0} q^{(2k+1)n^2} (F_{2n+1}+F_{2n-1})  \nonumber\\[5pt]
&\quad =\sum_{n_k\geq n_{k-1}\geq \cdots \geq n_1\geq 0} \frac{q^{2(n_1^2+\cdots +n_k^2)}\prod_{j=1}^{n_1} (1+3q^{2j-1}+q^{4j-2})}{(q^2;q^2)_{n_k-n_{k-1}}(q^2;q^2)_{n_{k-1}-n_{k-2}}\cdots (q^2;q^2)_{n_2-n_1}(q^2;q^2)_{2n_1}}.   \label{mainthm3-4}
\end{align}
\end{subequations}
\end{coro}

Specially, when $k=1$, the above four identities reduce to \eqref{Entry532}, \eqref{Entry534}, \eqref{x-1} and \eqref{x32}, respectively.

\section{The weak forms \eqref{baileypair2} and \eqref{baileypair3} of Bailey's Lemma}

In this section, we consider the applications of the weak forms \eqref{baileypair2} and \eqref{baileypair3} of Bailey's Lemma. Firstly, by inserting our key Bailey pair \eqref{baileypairV} into the second weak form \eqref{baileypair2}, it leads to the following result.

\begin{theorem} We have
\begin{equation}\label{RR2}
\sum_{n\geq 0} \frac{q^{n^2}(-q;q^2)_n\prod_{i=1}^n(1+2xq^{2i-1}+q^{4i-2})}{(q^2;q^2)_{2n}}
=\frac{(-q;q^2)_\infty}{(q^2;q^2)_\infty} \sum_{n\geq 0} q^{2n^2} \big(V_n(x)+V_{n-1}(x)\big).
\end{equation}
\end{theorem}

Employing the similar procedures as given in Section~4 and taking $x$ to be $-1, -\frac{1}{2}, 0, \frac{1}{2}, 1$ and $\frac{3}{2}$ (or equivalently, $-\frac{3}{2}$), respectively, the above identity reduces to the following  Rogers--Ramanujan type identities.

\begin{coro}\label{corb} We have
\begin{subequations}
\begin{align}
&\sum_{n\geq 0} \frac{q^{n^2}(q;q^2)_n}{(q^4;q^4)_{n}}=\frac{ (q^2;q^4)_\infty^2}{(q;q^2)_\infty},\label{b-1}\\[5pt]
&\sum_{n\geq 0} \frac{q^{n^2}(-q^3;q^6)_n}{(q^2;q^2)_{2n}}
=\frac{(-q;q)_\infty(-q^6;q^{12})_\infty}{(-q^2;q^4)_\infty},\label{b-12}\\[5pt]
&\sum_{n\geq 0} \frac{q^{n^2}(-q^2;q^4)_n}{(q;q^2)_n(q^4;q^4)_n}=\frac{(-q;q^2)_\infty(q^8;q^8)_\infty (q^{8};q^{16})_\infty }{(q^2;q^2)_\infty},\label{b0}\\[5pt]
&\sum_{n\geq 0} \frac{q^{n^2}(-q;q^2)_n(q^3;q^6)_n}{(q^2;q^2)_{2n}(q;q^2)_n}=\frac{(q^4;q^4)_\infty (q^{6};q^{12})_\infty }{(q;q)_\infty},\label{b12}\\[5pt]
&\sum_{n\geq 0} \frac{q^{n^2}(-q;q^2)_n^3}{(q^2;q^2)_{2n}}=\frac{(-q^2;q^4)_\infty^2 }{(q;q^2)_\infty},\label{b1}\\[5pt]
&\sum_{n\geq 0} \frac{q^{n^2}(-q;q^2)_n\prod_{i=1}^n(1+3q^{2i-1}+q^{4i-2})}{(q^2;q^2)_{2n}}
=\frac{(-q;q^2)_\infty}{(q^2;q^2)_\infty} \sum_{n\geq 0} q^{2n^2} (F_{2n+1}+F_{2n-1}).\label{b32}
\end{align}
\end{subequations}

\end{coro}

We remark  that  when $x=-1, 0, \frac{1}{2}$ and $\frac{3}{2}$, the summation on the right hand sides of \eqref{RR2}  can be deduced by setting $q\to q^{\frac{2}{3}}$ in \eqref{x-1}, \eqref{Entry532}, \eqref{Entry534} and \eqref{x32}, which leads to \eqref{b-1}, \eqref{b0}, \eqref{b12}, and \eqref{b32}, respectively.
Now, we show how to derive \eqref{b-12} and \eqref{b1} from \eqref{RR2}.

{\noindent \it Proof of \eqref{b-12} and \eqref{b1}.} Taking $x=-\frac{1}{2}$ in \eqref{RR2} and with the aid of \eqref{V-12}, we have
\begin{align*}
\sum_{n\geq 0} \frac{q^{n^2}(-q^3;q^6)_n}{(q^2;q^2)_{2n}}&= \frac{(-q;q^2)_\infty}{(q^2;q^2)_\infty} \sum_{n\geq 0} q^{2n^2} \big(V_n(-\frac{1}{2})+V_{n-1}(-\frac{1}{2})\big)\\
&=\frac{(-q;q^2)_\infty}{(q^2;q^2)_\infty} \Big(2\sum_{n\geq 0} q^{2(3n)^2}-\sum_{n\geq 0} q^{2(3n+1)^2}-\sum_{n\geq 0} q^{2(3n+2)^2} \Big)\\
&=\frac{(-q;q^2)_\infty}{(q^2;q^2)_\infty} \Big(\sum_{n=-\infty}^\infty q^{2(3n)^2}-\sum_{n=-\infty}^\infty q^{2(3n+1)^2} \Big)\\
&=\frac{(-q;q^2)_\infty}{(q^2;q^2)_\infty} \sum_{n=-\infty}^\infty q^{2n^2}e^{\frac{2n\pi i}{3}}\\
&=\frac{(-q;q^2)_\infty}{(q^2;q^2)_\infty} (-q^2e^{\frac{2\pi i}{3}},-q^2e^{-\frac{2\pi i}{3}},q^4;q^4)_\infty\\
&=\frac{(-q;q^2)_\infty}{(q^2;q^2)_\infty}\frac{(-q^6;q^{12})_\infty (q^4;q^4)_\infty}{(-q^2;q^4)_\infty},
\end{align*}
which completes the proof of identity \eqref{b-12} by simplification.

Taking $x=1$ in \eqref{RR2} and employing  \eqref{V1}, we obtain that
\begin{align*}
\sum_{n\geq 0} \frac{q^{n^2}(-q;q^2)_n^3}{(q^2;q^2)_{2n}}
&=\frac{(-q;q^2)_\infty}{(q^2;q^2)_\infty} \sum_{n\geq 0} q^{2n^2} \big(V_n(1)+V_{n-1}(1)\big) \\
&=\frac{(-q;q^2)_\infty}{(q^2;q^2)_\infty} \big(1+2\sum_{n\geq 1} q^{2n^2} \big)\\
&=\frac{(-q;q^2)_\infty}{(q^2;q^2)_\infty} \sum_{n=-\infty}^\infty  q^{2n^2},
\end{align*}
which implies \eqref{b1} by applying Jacobi's triple product identity \eqref{jacobi} and simplifying. \qed

For these identities, one can see that \eqref{b-1}
is contained in Slater's list \cite[(4)]{Sla52} with $q\to -q$ and  also can be found in \cite[P.10, (2.4.2)]{LSZ08}; the identity \eqref{b-12} is Entry 5.3.8 in Ramanujan's lost notebook \cite[P. 105]{AB09};
\eqref{b0} can be found  in \cite[P.21, (2.16.4)]{LSZ08} and \cite[P. 16, (5.5)]{Sil07}; \eqref{b12} is Entry 5.3.9 in \cite[P. 105]{AB09}. Specially, \eqref{b1}
seems to be new, which can be seen as a  missing member of modular 4 identities in Slater's list, see \cite[P. 153]{Sla52} and \cite[P. 11]{LSZ08}.

Next, we consider the application of the third weak form of Bailey's lemma. By fitting the Bailey pair \eqref{baileypairV} into the weak form \eqref{baileypair3} with $q\to q^2$, we deduce the following result.

\begin{theorem}\label{RRc} We have
\begin{equation}\label{RR3}
2\sum_{n\geq 0} \frac{(-1)^n\prod_{i=1}^n(1+2xq^{2i-1}+q^{4i-2})}{(q^4;q^4)_n}
=\frac{(q^2;q^4)_\infty}{(q^4;q^4)_\infty}\sum_{n\geq 0}(-1)^n q^{n^2} \big(V_n(x)+V_{n-1}(x)\big).
\end{equation}
\end{theorem}

Following the similar procedures as given in Corollary \ref{corb}, by setting $x=0, \frac{1}{2}$ (or equivalently, $-\frac{1}{2}$), $1$ (or equivalently, $-1$) and $\frac{3}{2}$ (or equivalently, $-\frac{3}{2}$) in \eqref{RR3}, we obtain the following Rogers--Ramanujan type identities as consequences of Theorem \ref{RRc}.

\begin{coro} We have
\begin{subequations}
\begin{align}
&2\sum_{n\geq 0} \frac{(-1)^n(-q;q^2)_n}{(q^2;q^2)_n}= (q, q^2, q^3;q^4)_\infty,\label{c0}\\
&2\sum_{n\geq 0} (-1)^n\frac{(-q^3;q^6)_n}{(q^4;q^4)_n(-q;q^2)_n}
=\frac{(q^2;q^4)_\infty^2(q^3;q^6)_\infty}{(q;q^2)_\infty},\label{c-12}\\
&2\sum_{n\geq 0} \frac{(-1)^n(-q;q^2)_n^2}{(q^4;q^4)_n}
=(q, q^2, q^3;q^4)_\infty^2, \label{c1} \\[5pt]
&2\sum_{n\geq 0} \frac{(-1)^n\prod_{i=1}^n(1+3q^{2i-1}+q^{4i-2})}{(q^4;q^4)_n}
=\frac{(q^2;q^4)_\infty}{(q^4;q^4)_\infty}\sum_{n\geq 0}(-1)^nq^{n^2}(F_{2n+1}+F_{2n-1}).\label{c32}
\end{align}
\end{subequations}

\end{coro}

Note that among the above identities, \eqref{c1} can be seen as a special case of Entry 5.3.10 in \cite[p. 106]{AB09} with $a=1$. Moreover, we can see that \eqref{c0} is similar with the modular 4 identities in Slater's list, see \cite[P. 153]{Sla52} and \cite[P. 11]{LSZ08}, and \eqref{c-12}  can be seen as  a missing member in Slater's list of the modular 12 identities, see also   \cite[P.24]{LSZ08}.

\section{The weak form \eqref{baileypair4} of Bailey's Lemma and Appell--Lerch series}

In this section,  we will study the application of the weak form \eqref{baileypair4} of Bailey's Lemma in deriving identities on Appell--Lerch serires.

Recall that Appell--Lerch series are of the following form
\begin{equation}\label{appler}
\sum_{n=-\infty}^\infty \frac{(-1)^{\ell n}q^{\ell n(n+1)/2}b^n}{1-aq^n},
\end{equation}
which was first studied by Appell \cite{App846} and Lerch \cite{Ler92}. After multiplying the series \eqref{appler} by the factor $a^{\ell/2}$ and viewing it as function in the variables $a$, $b$ and $q$, it is also refereed as an Appell function of level $\ell$.

By inserting the Bailey pair \eqref{baileypairV} into the weak form  \eqref{baileypair4} with $q\to q^2$, we obtain the following result, from which some identities involving Appell--Lerch series are derived.

\begin{theorem} We have
\begin{align}\label{mainthmd}
\sum_{n\geq 0} &\frac{q^{n^2+n}(-1;q^2)_n \prod_{i=1}^n(1+2xq^{2i-1}+q^{4i-2})}{(q^2;q^2)_{2n}}
\nonumber\\[5pt]
&\quad =\frac{(-q^2;q^2)_\infty}{(q^2;q^2)_{\infty}}\sum_{n\geq 0}\frac{(-1;q^2)_n q^{2n^2+n}}{(-q^2;q^2)_n} \big(V_n(x)+V_{n-1}(x)\big).
\end{align}

\end{theorem}

By taking $x=1$ (or equivalently, $x=-1$),  in the above identity, we obtain the following result.

\begin{coro} We have
\begin{equation}\label{d-1}
\sum_{n\geq 0} \frac{q^{n^2+n}(-1;q^2)_n(-q;q^2)_n^2}{ (q^{2};q^2)_{2n}}
=2\frac{(-q^2;q^2)_\infty}{(q^2;q^2)_{\infty}}\sum_{n=-\infty}^\infty \frac{ q^{2n^2+n}}{1+q^{2n}}.
\end{equation}
\end{coro}

\proof The left hand side of \eqref{d-1} can be obtained directly by setting $x=-1$ in \eqref{mainthmd}. For the summation on the right hand side of \eqref{mainthmd},  with the aid of \eqref{V-1},  we see that
\begin{align*}
& \sum_{n\geq 0}\frac{(-1;q^2)_n q^{2n^2+n}}{(-q^2;q^2)_n} \big(V_n(1)+V_{n-1}(1)\big)\\[5pt]
&\qquad= 1+2\sum_{n\geq 1}\frac{(-1;q^2)_n q^{2n^2+n}}{(-q^2;q^2)_n}\\[5pt]
&\qquad = 1+4\sum_{n\geq 1}\frac{ q^{2n^2+n}}{1+q^{2n}}\\
&\qquad = 2 \Big(\frac{1}{2}+\sum_{n\geq 1}\frac{  q^{2n^2+n}}{1+q^{2n}}+\sum_{n=-\infty}^{-1}\frac{ q^{2n^2+n}}{1+q^{2n}}\Big)\\
&\qquad = 2 \sum_{n=-\infty}^{\infty}\frac{  q^{2n^2+n}}{1+q^{2n}},
\end{align*}
which completes the proof. \qed

It is notable that the summation on the right hand side of the above identity is closely related to the mock theta function of order 2, which is given by
\[
\mu^{(2)}(q)=2\frac{(q;q^2)_\infty}{(q^2;q^2)_\infty} \sum_{n=-\infty}^\infty \frac{q^{2n^2+n}}{1+q^{2n}}=\sum_{n=0}^\infty \frac{(-1)^n(q;q^2)_nq^{n^2}}{(-q^2;q^2)_n^2},
\]
see \cite{McI07, McI18}.
From \eqref{d-1}, we obtain another expression of $\mu^{(2)}(q)$ as follows
\begin{equation}\label{mock2}
\mu^{(2)}(q)=\frac{(q;q^2)_\infty}{(-q^2;q^2)_\infty}
\sum_{n=0}^\infty \frac{(-1;q^2)_n(-q;q^2)_n^2 q^{n^2+n}}{ (q^{2};q^2)_{2n}}.
\end{equation}

In \eqref{mainthmd}, by substituting $x=-\frac{1}{2}$ (or equivalently, $x=\frac{1}{2}$), we obtain the following result.

\begin{coro} We have
\begin{equation}\label{d-2}
\sum_{n\geq 0} \frac{q^{n^2+n}(-1;q^2)_n(-q^3;q^6)_n }{ (q^{2};q^2)_{2n}(-q;q^2)_n}
=2\frac{(-q^2;q^2)_\infty}{(q^2;q^2)_{\infty}}\sum_{n=-\infty}^\infty \frac{e^{\frac{2n\pi i}{3}}q^{2n^2+n}}{1+q^{2n}}.
\end{equation}
\end{coro}

\proof It is direct to obtain the left hand side of \eqref{d-2} by setting $x=-\frac{1}{2}$ in \eqref{mainthmd}. For the summation on the right hand side, by using \eqref{V-12}, we have that
\begin{align*}
& \sum_{n\geq 0}\frac{q^{2n^2+n}(-1;q^2)_n }{(-q^2;q^2)_n} \Big(V_n\big(-\frac{1}{2}\big)+V_{n-1}\big(-\frac{1}{2}\big)\Big)\\[5pt]
&\qquad = 1+4\sum_{n\geq 1}\frac{ q^{2(3n)^2+3n}}{1+q^{6n}}-2\sum_{n\geq 0}\frac{ q^{2(3n+1)^2+3n+1}}{1+q^{2(3n+1)}}-2\sum_{n\geq 0}\frac{ q^{2(3n+2)^2+3n+2}}{1+q^{2(3n+2)}}\\[5pt]
&\qquad = 2\sum_{n=-\infty}^\infty \frac{ q^{2(3n)^2+3n}}{1+q^{6n}}-2\sum_{n\geq 0}\frac{ q^{2(3n+1)^2+3n+1}}{1+q^{2(3n+1)}}-2\sum_{n=-\infty}^{-1} \frac{ q^{2(3n+1)^2+3n+1}}{1+q^{2(3n+1)}}\\[5pt]
&\qquad =  2\sum_{n=-\infty}^{\infty} \frac{ q^{2(3n)^2+3n}}{1+q^{6n}}-2\sum_{n=-\infty}^{\infty} \frac{ q^{2(3n+1)^2+3n+1}}{1+q^{2(3n+1)}}\\[5pt]
&\qquad =  2 \sum_{n=-\infty}^\infty \frac{e^{\frac{2n\pi i}{3}}q^{2n^2+n}}{1+q^{2n}},
\end{align*}
where the last step follows by considering the
remainder classes of $n$ module $3$.  \qed

Taking  $x=0$ in \eqref{mainthmd} and then setting $q^2\to q$, we obtain the following identity involving Apell--Lerch series.

\begin{coro} We have
\begin{equation}\label{d-3}
\sum_{n\geq 0} \frac{(-1;q)_{n}(-q;q^2)_n q^{\frac{n^2+n}{2}}}{ (q;q)_{2n}}
=2\frac{(-q;q)_\infty}{(q;q)_{\infty}}\sum_{n=-\infty}^\infty \frac{(-1)^nq^{4n^2+n}}{1+q^{2n}}.
\end{equation}
\end{coro}

\proof When $x=0$ in \eqref{mainthmd}, by using \eqref{V-0}, the summation on the right hand side becomes
\begin{align*}
& \sum_{n\geq 0}\frac{q^{2n^2+n}(-1;q^2)_n }{(-q^2;q^2)_n} \big(V_n(0)+V_{n-1}(0)\big)\\[5pt]
&\qquad =1+4 \sum_{n\geq 1} \frac{ q^{2(4n)^2+4n}}{1+q^{8n}}-4\sum_{n\geq 0}\frac{ q^{2(4n+2)^2+4n+2}}{1+q^{2(4n+2)}}\\[5pt]
&\qquad = 2\sum_{n=-\infty}^{\infty} \frac{ q^{2(4n)^2+4n}}{1+q^{8n}}-2\sum_{n\geq 0}\frac{ q^{2(4n+2)^2+4n+2}}{1+q^{2(4n+2)}}-2\sum_{n=-\infty}^{-1} \frac{ q^{2(4n+2)^2+4n+2}}{1+q^{2(4n+2)}}\\[5pt]
&\qquad = 2\sum_{n=-\infty}^{\infty} \frac{ q^{2(4n)^2+4n}}{1+q^{8n}}-2\sum_{n=-\infty}^\infty \frac{ q^{2(4n+2)^2+4n+2}}{1+q^{2(4n+2)}}\\[5pt]
&\qquad = 2\sum_{n=-\infty}^{\infty} (-1)^n \frac{ q^{2(2n)^2+2n}}{1+q^{4n}}.
\end{align*}
Then the proof is complete by replacing  $q^2$ with $q$. \qed

When $x=\frac{3}{2}$ (or equivalently, $x=-\frac{3}{2}$)  in \eqref{mainthmd}, we obtain the following result.
\begin{coro} We have
\begin{equation}\label{d-4}
\sum_{n\geq 0} \frac{(-1;q^2)_n\prod_{i=1}^n (1+3q^{2i-1}+q^{4i-2})}{ (q^{2};q^2)_{2n}}
=2\frac{(-q^2;q^2)_\infty}{(q^2;q^2)_{\infty}} \sum_{n\geq 0} \frac{ q^{2n^2+n}}{1+q^{2n}}\big(F_{2n+1}+F_{2n-1}\big).
\end{equation}
\end{coro}

\proof  It is direct to obtain the left hand side of \eqref{d-4} by setting $x=\frac{3}{2}$ in \eqref{mainthmd}. For the summation on the right hand side, by using \eqref{V23}, we have that
\begin{align*}
& \sum_{n\geq 0}\frac{q^{2n^2+n}(-1;q^2)_n }{(-q^2;q^2)_n} \Big(V_n\big(\frac{3}{2}\big)+V_{n-1}\big(\frac{3}{2}\big)\Big)\\[5pt]
&\qquad =1+ 2\sum_{n\geq 1} \frac{ q^{2n^2+n}}{1+q^{2n}}\big(F_{2n+1}+F_{2n-1}\big)\\[5pt]
&\qquad = 2\sum_{n\geq 0} \frac{ q^{2n^2+n}}{1+q^{2n}}\big(F_{2n+1}+F_{2n-1}\big),
\end{align*}
which completes the proof. \qed

\section{Generalized Hecke--type Series}

In this section, we take identity \eqref{baileyid} which leads to our key Bailey pair into establishing identities related to the generalized Hecke--type  series involving indefinite quadratic forms.

Recall that a series is of Hecke--type if it has the following form
\[
\sum_{(n,j)\in D} (-1)^{H(n,j)} q^{Q(n,j)+L(n,j)},
\]
where $H$ and $L$ are linear forms, $Q$ is a quadratic form, and $D$ is some subset of $\mathbb{Z}\times \mathbb{Z}$ such that $Q(n,j)\geq 0$ for any $(n,j)\in D$. Hecke--type series have
received extensive attention since the study of Jacobi and Hecke, see, for example \cite{And841, HM14, Lov04, WY20, WY21}. In \cite{And19}, Andrews introduced the generalized Hecke--type series in which the restriction  $Q(n,j)\geq 0$ is removed. In the same paper, Andrews also stated the following result
\begin{equation}\label{heine}
\sum_{n\geq 0}\frac{q^{n^2+\alpha n}}{(q;q)_n(q;q)_{n+\beta}}=\frac{1}{(q;q)_\infty} \sum_{n\geq 0}\frac{(q^{\alpha-\beta};q)_n(-1)^nq^{\beta n+{n+1\choose 2}}}{(q;q)_n},
\end{equation}
which can be derived from Heine's second transformation \cite[p. 241, eq. (III.2)]{GR04} by taking $a=b=\frac{1}{\tau}$, $z=q^{\alpha+1}\tau^2$, $c=q^{\beta+1}$, and then letting $\tau\to 0$.

With the light of some special cases of Andrews' identity \eqref{heine}, from identity \eqref{baileyid}  we can deduce the following results on generalized Hecke--type series.

\begin{theorem}\label{heckeA} We have
\begin{align} \label{hecke1}
\sum_{n\geq 0} & \frac{q^{2n^2+2n}\prod_{i=1}^n (1+2xq^{2i-1}+q^{4i-2})}{(q^2;q^2)_{2n}}\nonumber \\
&=\frac{1}{(q^2;q^2)_\infty}
\sum_{n\geq 0} (-1)^n q^{ n^2+n} \sum_{j=0}^{\lfloor\frac{n}{2} \rfloor} q^{-j^2}\big(V_j(x)+V_{j-1}(x)\big).
\end{align}
\end{theorem}

\begin{proof} Denote the left hand side of \eqref{hecke1} by $L$ and $v_n(x)=V_n(x)+V_{n-1}(x)$ as given in the proof of Theorem \ref{basici}. By using identity   \eqref{baileyid}, we have
\begin{align*}
L&=\sum_{n\geq 0}  \frac{q^{2n^2+2n}}{(q^2;q^2)_{2n}}\sum_{j=0}^n q^{j^2}{2n \brack n-j}_2 v_j(x)\\
&=\sum_{j\geq 0}\sum_{n\geq 0} \frac{q^{2(n+j)^2+2(n+j)+j^2} v_j(x)}{(q^2;q^2)_n(q^2;q^2)_{n+2j}}\\
&=\sum_{j\geq 0} q^{3j^2+2j} v_j(x)\sum_{n\geq 0} \frac{q^{2n^2+2(2j+1)n}}{(q^2;q^2)_n(q^2;q^2)_{n+2j}}.
\end{align*}
By applying \eqref{heine} with $q\to q^2, \alpha=2j+1, \beta=2j$, and then diving the summation according to the parity of $n$, we get
\begin{align}
L&=\frac{1}{(q^2;q^2)_\infty}\sum_{j\geq 0} q^{3j^2+2j} v_j(x)\sum_{n\geq 0} (-1)^n q^{n^2+n+4nj}\nonumber \\
&=\frac{1}{(q^2;q^2)_\infty}\sum_{j\geq 0} q^{3j^2+2j} v_j(x)\sum_{n\geq 0} q^{4n^2+2n+8nj}(1-q^{4n+4j+2})\nonumber \\
&=\frac{1}{(q^2;q^2)_\infty}\sum_{n\geq 0} q^{4n^2+2n} (1-q^{4n+2}) \sum_{j=0}^n q^{-j^2} v_j(x)\nonumber \\
&=\frac{1}{(q^2;q^2)_\infty}\sum_{n\geq 0} q^{4n^2+2n} (1-q^{4n+2}) \sum_{j=0}^n q^{-j^2}(V_j(x)+V_{j-1}(x)),  \label{hecke1a}
\end{align}
which completes the proof by reconsidering the parity of the variable $n$.
\end{proof}

When $x$ is taken to be special values, we can obtain some identities on generalized Hecke--type series. Let us take $x=-1$ firstly. It is also equivalent to the case when $x=1$ and $q\to -q$.

\begin{coro} We have
\begin{align*}
\sum_{n\geq 0}  \frac{q^{2n^2+2n}(q;q^2)_n^2}{(q^2;q^2)_{2n}}&=\frac{1}{(q^2;q^2)_\infty}
\sum_{n\geq 0} (-1)^n q^{n^2+n}\sum_{j=-\lfloor\frac{n}{2} \rfloor}^{\lfloor\frac{n}{2} \rfloor } (-1)^jq^{-j^2}.
\end{align*}
\end{coro}

\proof Taking $x=-1$ in \eqref{hecke1}, we obtain
\begin{align*}
\sum_{n\geq 0} & \frac{q^{2n^2+2n}\prod_{i=1}^n (q;q^2)_n^2}{(q^2;q^2)_{2n}} =\frac{1}{(q^2;q^2)_\infty}
\sum_{n\geq 0} (-1)^n q^{ n^2+n} \sum_{j=0}^{\lfloor\frac{n}{2} \rfloor} q^{-j^2}\big(V_j(-1)+V_{j-1}(-1)\big).
\end{align*}
By using the special values of $V_n(x)$ at $x=-1$ \eqref{V-1}, we are led to
\begin{align*}
\sum_{n\geq 0}  \frac{q^{2n^2+2n}  (q;q^2)_n^2}{(q^2;q^2)_{2n}} & =\frac{1}{(q^2;q^2)_\infty}
\sum_{n\geq 0} (-1)^n q^{ n^2+n} \big(1+2\sum_{j=1}^{\lfloor\frac{n}{2} \rfloor} (-1)^j q^{-j^2}\big)\\
&=\frac{1}{(q^2;q^2)_\infty}
\sum_{n\geq 0} (-1)^n q^{ n^2+n} \sum_{j=-\lfloor\frac{n}{2} \rfloor}^{\lfloor\frac{n}{2} \rfloor} (-1)^j q^{-j^2},
\end{align*}
which completes the proof. \qed

By setting $x=0$  in \eqref{hecke1} and then substituting $q^2$ by $q$, we obtain the following result.

\begin{coro} We have
\begin{align*}
\sum_{n\geq 0}  \frac{q^{n^2+n}(-q;q^2)_n}{(q;q)_{2n}}&=\frac{1}{(q;q)_\infty} \sum_{n\geq 0} (-1)^nq^{\frac{n(n+1)}{2}}\sum_{j=-\lfloor\frac{n}{4} \rfloor}^{\lfloor\frac{n}{4} \rfloor} (-1)^jq^{-2j^2}.
\end{align*}
\end{coro}

\proof  To be more direct, we start from the  expression \eqref{hecke1a}.  By substituting $x=0$ in \eqref{hecke1a}, we obtain
\begin{align*}
\sum_{n\geq 0} & \frac{q^{2n^2+2n}  (-q^2;q^4)_n}{(q^2;q^2)_{2n}} =\frac{1}{(q^2;q^2)_\infty}\sum_{n\geq 0} q^{4n^2+2n} (1-q^{4n+2}) \sum_{j=0}^n q^{-j^2}(V_j(0)+V_{j-1}(0)).
\end{align*}
By employing \eqref{V-0}, it leads to that
\begin{align*}
\sum_{n\geq 0} & \frac{q^{2n^2+2n}  (-q^2;q^4)_n}{(q^2;q^2)_{2n}}\\
 &=\frac{1}{(q^2;q^2)_\infty}\sum_{n\geq 0} q^{4n^2+2n} (1-q^{4n+2}) \Big(1+2\sum_{j=1}^{\lfloor\frac{n}{4} \rfloor} q^{-(4j)^2}-2\sum_{j=0}^{\lfloor\frac{n-2}{4} \rfloor} q^{-(4j+2)^2}\Big)\\
&=\frac{1}{(q^2;q^2)_\infty}\sum_{n\geq 0} q^{4n^2+2n} (1-q^{4n+2}) \Big(\sum_{j=-\lfloor\frac{n}{4} \rfloor }^{\lfloor\frac{n}{4} \rfloor} q^{-(4j)^2}-\sum_{j=-\lfloor\frac{n-2}{4} \rfloor -1}^{\lfloor\frac{n-2}{4} \rfloor} q^{-(4j+2)^2}\Big)\\
&=\frac{1}{(q^2;q^2)_\infty}\sum_{n\geq 0} q^{4n^2+2n} (1-q^{4n+2}) \sum_{j=-\lfloor\frac{n}{2} \rfloor }^{\lfloor\frac{n}{2} \rfloor} (-1)^jq^{-(2j)^2}\\
&=\frac{1}{(q^2;q^2)_\infty}\sum_{n\geq 0} (-1)^nq^{n(n+1)}  \sum_{j=-\lfloor\frac{n}{4} \rfloor }^{\lfloor\frac{n}{4} \rfloor} (-1)^jq^{-(2j)^2},
\end{align*}
where the last step follows by taking the parity of $n$ into consideration. Then the proof is complete by setting $q^2\to q$. \qed

By setting $x=-\frac{1}{2}$  in \eqref{hecke1}, we obtain
the following identity. It is also equivalent to the result given by taking $x=-\frac{1}{2}$  and then $q\to -q$.

\begin{coro} We have
\[
\sum_{n\geq 0}  \frac{q^{2n^2+2n}(-q^3;q^6)_n}{(q^2;q^2)_{2n}(-q;q^2)_n}=\frac{1}{(q^2;q^2)_\infty}\sum_{n\geq 0} q^{n(n+1)}\sum_{j=-\lfloor \frac{n}{2}\rfloor}^{\lfloor \frac{n}{2}\rfloor}  e^{\frac{2\pi ij}{3}}q^{-j^2}.
\]
\end{coro}

\proof As in the above corollary, we take  $x=-\frac{1}{2}$ in \eqref{hecke1a}. With the help of \eqref{V-12}, we get
\begin{align*}
\sum_{n\geq 0} & \frac{q^{2n^2+2n}  (-q^3;q^6)_n}{(q^2;q^2)_{2n}(-q;q^2)_n} \\
&=\frac{1}{(q^2;q^2)_\infty}\sum_{n\geq 0} q^{4n^2+2n} (1-q^{4n+2}) \sum_{j=0}^n q^{-j^2}\big(V_j(-\frac{1}{2})+V_{j-1}(-\frac{1}{2})\big)\\
&=\frac{1}{(q^2;q^2)_\infty}\sum_{n\geq 0} q^{4n^2+2n} (1-q^{4n+2}) \Big(1+2\sum_{j=1}^{\lfloor \frac{n}{3}\rfloor} q^{-(3j)^2}- \sum_{j=0}^{\lfloor \frac{n-1}{3}\rfloor} q^{-(3j+1)^2}-\sum_{j=0}^{\lfloor \frac{n-2}{3}\rfloor} q^{-(3j+2)^2}\Big)\\
&=\frac{1}{(q^2;q^2)_\infty}\sum_{n\geq 0} q^{4n^2+2n} (1-q^{4n+2})\bigg(\sum_{j=-\lfloor \frac{n}{3}\rfloor}^{\lfloor \frac{n}{3}\rfloor} q^{-(3j)^2}-\sum_{j=-\lfloor \frac{n+1}{3}\rfloor}^{\lfloor \frac{n-1}{3}\rfloor} q^{-(3j+1)^2}\bigg)\\
&=\frac{1}{(q^2;q^2)_\infty}\sum_{n\geq 0} q^{4n^2+2n} (1-q^{4n+2}) \sum_{j=-n}^{n} e^{\frac{2\pi ij}{3}}q^{-j^2}\\
&=\frac{1}{(q^2;q^2)_\infty}\sum_{n\geq 0} q^{n(n+1)}\sum_{j=-\lfloor \frac{n}{2}\rfloor}^{\lfloor \frac{n}{2}\rfloor}  e^{\frac{2\pi ij}{3}}q^{-j^2},
\end{align*}
where the last step is derived  by taking the parity of $n$ into consideration. \qed

Applying Andrews' rewritten form \eqref{heine} of Heine's transformation formula, we also obtain the following result on the generalized Hecke--type series.

\begin{theorem} We have
\begin{align} \label{hecke2}
\sum_{n\geq 0} & \frac{q^{2n^2-2n}\prod_{i=1}^n (1+2xq^{2i-1}+q^{4i-2})}{(q^2;q^2)_{2n-1}}\nonumber \\
&=\frac{1}{(q^2;q^2)_\infty}
\sum_{n\geq 0} q^{4n^2-2n}(1-q^{12n+6})\sum_{j=0}^n q^{-j^2}\big(V_j(x)+V_{j-1}(x)\big),
\end{align}
in which the summand on the left hand side is equal to zero when $n=0$.
\end{theorem}

\begin{proof} Denote the left hand side of \eqref{hecke2} by $L$ and let $v_j(x)=V_j(x)+V_{j-1}(x)$ for $j\geq 0$. Using identity \eqref{baileyid}, we obtain
\begin{align*}
L&=\sum_{n\geq 0}  \frac{q^{2n^2-2n}}{(q^2;q^2)_{2n-1}}\sum_{j=0}^n q^{j^2}{2n \brack n-j}_2 v_j(x)\\
&=\sum_{j\geq 0}\sum_{n\geq 0} \frac{q^{2(n+j)^2-2(n+j)+j^2} v_j(x)(1-q^{4n+4j})}{(q^2;q^2)_n(q^2;q^2)_{n+2j}}\\
&=\sum_{j\geq 0} q^{3j^2-2j} v_j(x)\sum_{n\geq 0} \frac{q^{2n^2+2(2j-1)n}(1-q^{4n+4j})}{(q^2;q^2)_n(q^2;q^2)_{n+2j}}\\
&=\sum_{j\geq 0} q^{3j^2-2j} v_j(x)\bigg(\sum_{n\geq 0} \frac{q^{2n^2+2(2j-1)n}}{(q^2;q^2)_n(q^2;q^2)_{n+2j}}-q^{4j}\sum_{n\geq 0} \frac{q^{2n^2+2(2j+1)n}}{(q^2;q^2)_n(q^2;q^2)_{n+2j}}\bigg)
\end{align*}
By applying \eqref{heine} with $q\to q^2, \beta=2j$, and substituting $\alpha$ by $2j-1$ and $2j+1$, respectively, the above result becomes
\begin{align*}
L&=\frac{1}{(q^2;q^2)_\infty}\sum_{j\geq 0} q^{3j^2-2j} v_j(x)
\bigg(\sum_{n\geq 0} \frac{(q^{-2};q^2)_n (-1)^n q^{n^2+n+4nj}}{(q^2;q^2)_n}-q^{4j}\sum_{n\geq 0} (-1)^n q^{n^2+n+4nj}\bigg)\\
&=\frac{1}{(q^2;q^2)_\infty}\sum_{j\geq 0} q^{3j^2-2j} v_j(x)\Big(1+q^{4j}-q^{4j}\sum_{n\geq 0} (-1)^n q^{n^2+n+4nj}\Big)\\
&=\frac{1}{(q^2;q^2)_\infty}\sum_{j\geq 0} q^{3j^2-2j} v_j(x)\Big(1-q^{4j}\sum_{n\geq 1} (-1)^n q^{n^2+n+4nj}\Big) \end{align*}
By dividing the above sum on $n$ into two parts according to the parity of $n$, it implies that
\begin{align*}
L&=\frac{1}{(q^2;q^2)_\infty}\sum_{j\geq 0} q^{3j^2-2j} v_j(x)\Big(1-\sum_{n\geq 0}  q^{(2n+2)^2+(2n+2)+4(2n+2)j+4j}+\sum_{n\geq 1}  q^{(2n-1)^2+(2n-1)+4(2n-1)j+4j}\Big) \\
&=\frac{1}{(q^2;q^2)_\infty}\sum_{j\geq 0} q^{3j^2-2j} v_j(x)\Big(\sum_{n\geq 0}  q^{4n^2-2n+8nj}-\sum_{n\geq 0}  q^{4n^2+10n+8nj+12j+6}\Big)\\
&=\frac{1}{(q^2;q^2)_\infty}\sum_{j\geq 0} q^{3j^2-2j} v_j(x) \sum_{n\geq 0}  q^{4n^2-2n+8nj}(1-q^{12n+12j+6})\\
&=\frac{1}{(q^2;q^2)_\infty}\sum_{n\geq 0} q^{4n^2-2n}(1-q^{12n+6}) \sum_{j=0}^n q^{-j^2} v_j(x),
\end{align*}
which completes the proof.
\end{proof}

Following the similar procedures as given in deducing the corollaries of Theorem \ref{heckeA}, and taking  $x=-1$ (or equivalently, $x=1$), $x=0$ (with $q^2\to q$) and $x=-\frac{1}{2}$ (or equivalently, $x=\frac{1}{2}$) in \eqref{hecke2},  respectively, we obtain the following results on generalized Hecke--type series.

\begin{coro} We have
\begin{align*}
&\sum_{n\geq 0} \frac{q^{2n^2-2n}(q;q^2)_n^2}{(q^2;q^2)_{2n-1}}
=\frac{1}{(q^2;q^2)_\infty}
\sum_{n\geq 0} q^{4n^2-2n}(1-q^{12n+6})\sum_{j=-n}^n (-1)^jq^{-j^2},\\[5pt]
&\sum_{n\geq 0} \frac{q^{n^2-n}(q;q^2)_n}{(q;q)_{2n-1}}
=\frac{1}{(q;q)_\infty}
\sum_{n\geq 0} q^{2n^2-n}(1-q^{6n+3})\sum_{j=-\lfloor\frac{n}{2} \rfloor}^{\lfloor\frac{n}{2} \rfloor} (-1)^jq^{-2j^2},\\[5pt]
&\sum_{n\geq 0} \frac{q^{2n^2-2n}(-q^3;q^6)_n}{(q^2;q^2)_{2n-1}(-q;q^2)_n}=\frac{1}{(q^2;q^2)_\infty}
\sum_{n\geq 0} q^{4n^2-2n}(1-q^{12n+6})\sum_{j=-n}^{n} e^{\frac{2\pi ij}{3}}q^{-j^2}.
\end{align*}

\end{coro}

\section{Relations with Andrews' result}

As the last remark,  by comparing our result \eqref{baileyid} with Andrews' identity \eqref{Andrewidentity}, and using the orthogonality of Chebyshev polynomials of the third kind, we obtain an identity on $q$-binomial coefficients.

It's known that the orthogonality on $V_n(x)$ is given as follows
\begin{equation}\label{ortho}
\int_{-1}^1 \sqrt{\frac{1+x}{1-x}} V_n(x) V_m(x) \textup{d} x=
\left\{
  \begin{array}{ll}
    0, & \hbox{if $m\neq n$,} \\[5pt]
    \pi, & \hbox{if $m=n \geq 0$,}
  \end{array}
\right.
\end{equation}
see, for example, \cite[Sec. 4.2.2]{MH03}

\begin{theorem} We have
\begin{equation}\label{qbi}
\sum_{j=-n-1}^n q^{2j^2+j}{2n+1 \brack n-j}_{q^2}^2 = (1+q^{2n+1}){4n+1 \brack 2n}.
\end{equation}
\end{theorem}

\proof Setting $n$ by $2n$ in Andrews' identity \eqref{Andrewidentity}, we obtain that
\begin{equation}\label{Andrews2n}
\prod_{j=1}^{2n} (1+2xq^j+q^{2j})=\sum_{j=0}^{2n} q^{j+1\choose 2} V_j(x) {4n+1\brack 2n-j}.
\end{equation}
Apparently,  the left hand side of the above identity  can be rewritten as
\[
\prod_{j=1}^{2n} (1+2xq^j+q^{2j})=\prod_{j=1}^{n} (1+2xq^{2j}+q^{4j})\prod_{j=1}^{n} (1+2xq^{2j-1}+q^{4j-2}).
\]
By replacing the left hand side of the above identity with \eqref{Andrews2n}, and the two product terms on the right hand side  by \eqref{Andrewidentity} (with $q\to q^2$) and  \eqref{baileyid}, respectively, it turns out that
\[
\sum_{j=0}^{2n} q^{j+1\choose 2} V_j(x) {4n+1\brack 2n-j}=\sum_{j=0}^{n} q^{j^2+j} {2n+1\brack n-j}_{q^2} V_j(x)\cdot \sum_{j=0}^{n} q^{j^2}{2n \brack n-j}_{q^2} \big(V_j(x)+V_{j-1}(x)\big).
\]
Then  multiplying both hand sides of the above identity by $\sqrt{\frac{1+x}{1-x}}$, and calculating  the integrals on $x\in [-1,1]$ by employing the orthogonality property \eqref{ortho}, we are led to
\begin{align*}
{4n+1\brack 2n}&=\sum_{j=0}^{n} q^{2j^2+j} {2n+1\brack n-j}_{q^2} {2n \brack n-j}_{q^2}  +\sum_{j=0}^{n-1} q^{j^2+j+(j+1)^2} {2n+1\brack n-j}_{q^2} {2n \brack n-j-1}_{q^2}\\
&=\sum_{j=0}^{n} q^{2j^2+j} {2n+1\brack n-j}_{q^2}^2 \frac{(1+q^{2j+1})}{(1+q^{2n+1})}\\
&=\frac{1}{1+q^{2n+1}}\bigg(\sum_{j=0}^{n} q^{2j^2+j} {2n+1\brack n-j}_{q^2}^2 +\sum_{j=0}^{n} q^{(j+1)(2j+1)} {2n+1\brack n-j}_{q^2}^2\bigg)\\
&=\frac{1}{1+q^{2n+1}}\bigg(\sum_{j=0}^{n} q^{2j^2+j} {2n+1\brack n-j}_{q^2}^2 +\sum_{j=1}^{n+1} q^{j(2j-1)} {2n+1\brack n-j+1}_{q^2}^2\bigg)\\
&=\frac{1}{1+q^{2n+1}}\sum_{j=-n-1}^{n} q^{2j^2+j} {2n+1\brack n-j}_{q^2}^2,
\end{align*}
which completes the proof by further multiplying both hand sides by $1+q^{2n+1}$. \qed

Note that identity \eqref{qbi}  can be seen as a finite form of Jacobi's triple product identity \eqref{jacobi} with $q\to q^4$ and  $z\to -q^3$. In fact, by taking $n\to \infty$ in \eqref{qbi} and then simplifying,  we obtain that
\[
\sum_{j=-\infty}^\infty q^{2j^2+j}=(-q,-q^3,q^4;q^4)_\infty.
\]

\vskip 15pt
\noindent {\small {\bf Acknowledgments.}  This work is supported by the Fundamental Research Funds
for the Central Universities and the National Science Foundation of China (No. 12071235).

\end{document}